\documentclass[a4paper,12pt]{article}

\usepackage{times, url}
\usepackage{amsmath,amssymb,amsthm,amsfonts,bm,float,cite}
\usepackage{booktabs}
\usepackage[table]{xcolor}
\usepackage{tabularx}
\usepackage{boldline}
\usepackage[utf8]{inputenc}
\usepackage{microtype}
\usepackage{esvect}

\newtheorem{definition}{Definition}[section]

\newtheorem{theorem}[definition]{Theorem}
\newtheorem{lemma}[definition]{Lemma}
\newtheorem{corollary}[definition]{Corollary}

\usepackage[none]{hyphenat}
\usepackage[center]{caption}
\usepackage{mathtools}

\usepackage[pdftex]{graphicx}
\usepackage{mathtools}
\DeclareGraphicsExtensions{.png}
\makeatletter
\renewcommand{\@biblabel}[1]{[#1]\hfill}
\makeatother

\begin{document}
\setcounter{page}{1}

\begin{center}
{\LARGE \bf Optimal cycles enclosing all the nodes of a $k$-dimensional hypercube \\[4mm]}
\vspace{8mm}

{\Large \bf Roberto Rinaldi$^1$ and Marco Rip\`a$^2$}
\vspace{3mm}

$^1$ Independent Researcher \\ 
Rome, Italy \\
e-mail: \url{rinaldi.1668616@studenti.uniroma1.it}
\vspace{3mm}

$^2$ World Intelligence Network\\ 
Rome, Italy\\
e-mail: \url{marco.ripa@mensa.it}

\end{center}
\vspace{8mm}


\sloppy \noindent {\bf Abstract:} We solve the general problem of visiting all the $2^k$ nodes of a $k$-dimensional hypercube by using a polygonal chain that has minimum link-length and we show that this optimal value is given by $h(2,k):=3 \cdot 2^{k-2}$ if and only if $k \in \mathbb{N}-\{0,1\}$. Furthermore, for any $k$ above one, we constructively prove that it is possible to visit once and only once all the aforementioned nodes, $H(2,k):=\{\{0,1\} \times \{0,1\} \times \dots \times \{0,1\}\} \subset \mathbb{R}^k$, with a cycle (i.e., a closed path) having only $3 \cdot 2^{k-2}$ links.

\noindent {\bf Keywords:} Euclidean space, Optimization, Link distance, Combinatorics.

\noindent {\bf 2020 Mathematics Subject Classification:} 05C38 (Primary); 05C12, 91A43 (Secondary).
\vspace{2mm}


\section{Introduction} \label{sec:Intr}
Given $k \in \mathbb{N}-\{0,1\}$, let the finite subset of $2^k$ points belonging to the Euclidean space $\mathbb{R}^k$ be defined as $H(2,k):=\{\{0,1\} \times \{0,1\} \times \dots \times \{0,1\}\} \subset \mathbb{R}^k$, where the symbol ``$\times$'' denotes the well-known cartesian product. Then, the problem of joining all the nodes of $H(2,k)$ with a connected set of segments having minimum cardinality is equivalent to asking ourselves which is the minimum-link covering tree embedding all the nodes of a $k$-dimensional hypercube. This question is not an open problem, since Dumitrescu and T\'oth (see Reference \cite{Dumitrescu:12}, Figure 2), in 2014, easily showed that it is sufficient to connect all the $2^{k-1}$ pairs of opposite nodes with as many segments so that all of them (i.e., exactly $2^{k-1}$ line segments) meet in the center, $\textnormal{C} \equiv \left(\frac{1}{2}, \frac{1}{2}, \dots, \frac{1}{2} \right)$.

Now, it is obvious to understand why the number of (line) segments of the above-mentioned minimum-link covering tree for $H(2,k)$ cannot match the link-length of any polygonal chain covering the same set of $2^k$ nodes.

Thus, we are interested in solving a multidimensional \textit{thinking outside the box} problem, quite similar to the $k$-dimensional generalization of the infamous \textit{nine dots puzzle} \cite{Loyd:7, Chein:20, Kershaw:21} which was solved in 2020 by Rip\`a \cite{Ripa:11}: the crucial special case of finding a minimum-link polygonal chain covering any given $H(2,k):=\{\{0,1\} \times \{0,1\} \times \dots \times \{0,1\}\}$, a fascinating challenge belonging to the general problem of finding a minimum-link covering trail for every set $H(n,k)$ (see References \cite{OEIS:31, Ripa:23, Ripa:22}).

In order to introduce the original results of the present paper (Section \ref{sec:2}), let us give a few definitions first.

\begin{definition} \label{def1.1}
Let $\mathcal{P}(m):=(\rm{S_1})$-$(\rm{S_2})$-$\dots$-$({\rm{S}}_{m+1})$, a polygonal chain consisting of $m$ links, be well-defined through the sequence of its $m+1$ vertices, so $\mathcal{P}(m) \equiv \{\overline{{\rm{S_1 S_2}}} \cup \overline{{\rm{{S_2 S_3}}}} \cup \dots \cup \overline{{\rm{S}}_{m}{\rm{S}}_{m+1}}\}$. \linebreak
In particular, for any $d \in \mathbb{Z}^+ : d \leq m+1$, let the $d$-th vertex of $\mathcal{P}(m) \subset \mathbb{R}^k$ be univocally identified by the $k$-tuple $(x_1, x_2, \dots, x_k)$ (e.g., given $\mathcal{P}(3):=(0,0)$-$(1,0)$-$(1,1)$-$(0,1)$ and $d : d=2$, we have $x_1(S_d)=1 \wedge x_2(S_d)=0$, since $(x_1(S_2), x_2(S_2)) \equiv (1,0)$ for the aforementioned minimum-link polygonal chain covering $H(2,2):=\{{\{0,1\} \times \{0,1\}}\}$).
\end{definition}

\begin{definition} \label{def1.2}
Accordingly to Definition \ref{def1.1}, let $h(2,k)$ denote the link-length of the minimum-link polygonal chain $\mathcal{P}(h(2,k)):=(\rm{S_1})$-$(\rm{S_2})$-$\dots $-$({\rm{S}}_{h(2,k)+1})$ visiting all the nodes $H(2,k):=\{\{0,1\} \times \{0,1\} \times \dots \times \{0,1\}\}$ of the $k$-dimensional hypercube $\{[0, 1] \times [0, 1] \times \dots \times [0, 1]\}$.
\end{definition}

\begin{definition} \label{def1.3}
Let $\mathcal{P}(h(2,k))$ be a (possibly self-intersecting) path if there is not any element of the set $H(2,k)$ which belongs to more than one link of $\mathcal{P}(h(2,k))$. Then, let $\mathcal{P}(h(2,k))$ be a cycle if it is a path such that $(\rm{S_1}) \equiv$ $({\rm{S}}_{m+1})$ (i.e., we call a cycle any closed path).
Furthermore, we define as “perfect covering cycle", $\mathcal{\bar{C}}(h(2,k))$, any closed path such that no element of the set $\{\rm{S_\textit{j}} \mid \textit{j}=1, 2, \dots, \textit{m}+1\}$ belongs to more than two links of the given covering path for $H(2,k)$.
\end{definition}

Lastly, for clarity sake, let us specify that we will use \textit{vertices} and \textit{links} when we are referring to the turning points (usually we consider Steiner points which do not belong to $\{[0, 1] \times [0, 1] \times \dots \times [0, 1]\} \subset \mathbb{R}^k$) of the polygonal chains that we are taking into account, whereas we will prefer \textit{nodes} and \textit{edges} for the respective subsets of points entirely belonging to given $k$-dimensional hypercube.

Since this paper aims to find polygonal chains, embedding $H(2,k)$, which are optimal with respect to the number of line segments, we immediately point out that any covering cycle for $H(2,k)$ is a covering path for the same set of points, and it is also a polygonal chain covering all the $2^k$ nodes of the hypercube $\{[0, 1] \times [0, 1] \times \dots \times [0, 1]\} \subset \mathbb{R}^k$.

Now, a constructive proof of the existence of covering cycles for $H(2,2)$ and $H(2,3)$ having link-length of only $h(2,2)=3$ and $h(2,3)=6$, respectively, has already been shown in Reference \cite{Ripa:13} (e.g., the aforementioned paper, see pages 163–164 and Figures 7 to 9, provides an optimal upper bound for any $k \in \{2,3\}$ that is also achievable by taking into account only covering cycles, instead of generic polygonal chains, since $H(2,2) \subset \mathcal{P}(3)=\left(\frac{1}{2}, \frac{3}{2}\right)$-$(2, 0)$-$(-1, 0)$-$\left(\frac{1}{2}, \frac{3}{2}\right)$ and $H(2,3) \subset \mathcal{P}(6)=\left(\frac{1}{2}, \frac{1}{2}, \frac{3}{2}\right)$-$(2,2,0)$-$(-1,-1,0)$-$\left(\frac{1}{2}, \frac{1}{2}, \frac{3}{2}\right)$-$(2,-1,0)$-$(-1,2,0)$-$\left(\frac{1}{2}, \frac{1}{2}, \frac{3}{2}\right)$).

Thus, from here on, let us assume that $k \geq 2$ is given. The goal of the present research paper is to show that $h_l(2,k)=3 \cdot 2^{k-2}$ is a valid lower bound for any polygonal chain visiting all the elements $H(2,k)$ and to constructively prove that the aforementioned lower bound is equal to the upper bound $h_u(2,k)=3 \cdot 2^{k-2}$ \cite{OEIS:30} which returns the minimum link-lenght of any perfect covering cycle for the same set of nodes.

It follows that we are going to prove that $h_u(2,k)=h_l(2,k)$ by providing optimal covering cycles consisting of $3 \cdot 2^{k-2}$ links for any $k \in \mathbb{N}-\{0,1\}$, and this result will be shown in the next section.


\section{Main Result} \label{sec:2}

In order to prove that $h_l(2,k)=3 \cdot 2^{k-2}$ holds for any $k \in \mathbb{N}-\{0,1\}$, we will introduce the following lemma.

\begin{lemma}\label{lemma 1}
For any $k \in \mathbb{N}-\{0,1\}$, it is not possible to visit more than $4$ distinct elements of the set $H(2,k):=\{\{0,1\} \times \{0,1\} \times \dots \times \{0,1\}\} \subset \mathbb{R}^k$ by using a polygonal chain with $3$ links.
\end{lemma}

\begin{proof}

We prove Lemma \ref{lemma 1} by studying the generic trail $\mathcal{P}(3) \equiv \{\overline{\rm{S_1 S_2}} \cup \overline{\rm{S_2 S_3}} \cup \overline{\rm{S_3 S_4}}\}$ that passes through $4$ (distinct) nodes of $H(2, k)$. Then, we will show that there is no choice for these four nodes (and also for the considered Steiner points) that implies the existence of a fifth node belonging to $\mathcal{P}(3)$. We will start by demonstrating that there is not a trail $\mathcal{P}(2)$ that, passing through at least $3$ nodes, visits a fourth node. Considering that we need to pass through at least $3$ nodes, and be able to visit at most $2$ nodes with the first segment, we can impose that $\overline{{\rm S}_1{\rm S}_2}$ passes through $2$ nodes. Although, for convenience, we will impose ${\rm S}_1 \equiv {\rm V}_1 \equiv {\rm O}$, the obtained result can be extended to any choice of ${\rm S}_1$.

\vspace{2mm}

Let ${\rm S}_j$ be the origin of a given half-line $q_j$. Given a parameter $t_j\in \mathbb{R}:t_j\geq 0$, the corresponding parametric equation is of the form $q_j = {\rm S}_j + t_j \cdot {\rm \vv{{\rm S}_j{\rm V}_{j+1}}}$, where ${\rm V}_j$ indicates the $i$-th node of $H(2,k)$ visited by $\mathcal{P}(m)$. In this way, for $t_j = 1$, each of the Cartesian coordinates of the nodes must assume the value $0$ or $1$. Our goal is to show that this happens only for $t_j = 1$, and if this occurs for other values of $t_j$, then we want to show that the visited node is a node already visited previously.

Since the Steiner point ${\rm S}_{j+1}$ belongs to the considered half-line, let us denote by $\bar{t_j}$ the value of the parameter $t_j$ such that ${\rm S}_{j+1}={\rm S}_{j}+ \bar{t_j} \cdot {\rm \vv{{\rm S}_j{\rm V}_{j+1}}}$ (i.e., ${\rm S}_{j+1}:={\rm S}_{j+1}(\bar{t_j})$).

\vspace{2mm}

The generic point ${\rm S}_2$ is obtained as the last endpoint of a segment passing through ${\rm V}_2$,
\begin{equation}\label{S2}
    {\rm S}_2 = {\rm S}_1 + \bar{t_1} \cdot \vv{{\rm S}_1{\rm V}_2}.
\end{equation}

Now, from here on, we will indicate the $i$-th coordinate of the generic point ${\rm P}$, belonging to the Euclidean space $\mathbb{R}^k$, as $x_i({\rm P})$ (e.g., ${\rm P} \equiv (x_1({\rm P}), x_2({\rm P}), \dots, x_k({\rm P}))$).

\vspace{2mm}

Consequently, from Equation (\ref{S2}), it follows that 
\begin{equation}
    x_i\left({\rm S}_2\right)= \bar{t_1} \cdot x_i\left({\rm V}_2\right),
\end{equation} where $\bar{t_1} \geq1$.

\vspace{4mm}

Similarly, we will have that
\vspace{2mm}
\begin{equation}
    {\rm S}_3 = {\rm S}_2 + \bar{t_2} \cdot \vv{ {\rm S}_2{\rm V}_3}.
\end{equation}

Hence,
\begin{equation}
    \begin{split}
        x_i\left({\rm S}_3\right)&= \bar{t_1} \cdot x_i\left({\rm V}_2\right) + \bar{t_2} \cdot \left(x_i\left({\rm V}_3\right)-\bar{t_1} \cdot x_i \cdot \left({\rm V}_2\right)\right)\\
        &= \bar{t_1} \cdot x_i\left({\rm V}_2\right) \cdot \left(1-\bar{t_2}\right) + \bar{t_2} \cdot x_i\left({\rm V}_3\right),
    \end{split}
\end{equation} where $\bar{t_2} \geq 1$.

Let us consider the segment $\overline{{\rm S}_2{\rm S}_3}$. By disregarding the node ${\rm V}_3$, obtained by imposing $t_2 = 1$, we verify that it does not exist any node ${\rm V_j}$ of $H(2,k)$, belonging to $\overline{{\rm S}_2{\rm S}_3}$, which has not been previously visited.

Thus, for $i<k$, we need to study all the $x_i({\rm V}_j)$ equations, showing that there are no solutions such that $0<t_2<1$. It is not necessary to continue beyond point ${\rm V}_3$, studying solutions for $t_2>1$. If we encounter a node of the set $H(2,k)$ after point ${\rm V}_3$, then we will necessarily have to study also the case in which point ${\rm V}_3$ represents the furthest node and, with $t_2 <1$, we will find the point studied previously.

As a result, we have to study the above-mentioned $x_i({\rm V}_j)$ equations,
\begin{equation}
    \begin{cases}
    \bar{t_1} \cdot x_i({\rm V}_2) + t_2 \cdot (x_i({\rm V}_3)-\bar{t_1} \cdot x_i({\rm V}_2))=x_i({\rm V}_j) \\
    \bar{t_1}>1 \hspace{7.5cm}.\\
    0<t_2<1
    \end{cases}
\end{equation}

Hence,
\begin{align}
    &x_i({\rm V}_2)=x_i({\rm V}_3)=0 \Rightarrow t_2\in \mathbb{R}:0<t_2<1,\\
    &x_i({\rm V}_3)=0 \Rightarrow t_2=\frac{\bar{t_1}-1}{\bar{t_1}}.
\end{align}

Consequently, all the solutions imply $x_i({\rm V}_3)=0$. If $x_i({\rm V}_3)=0$ holds for all $i : i<k$, then $({\rm V}_3)=({\rm V}_1)$.

Now, we are finally ready to study the generic trail $\mathcal{P}(3) \equiv \{\overline{\rm{S_1 S_2}} \cup \overline{\rm{S_2 S_3}} \cup \overline{\rm{S_3 S_4}}\}$.

Thanks to the results discussed above, we know that if $\overline{\rm{S_2 S_3}}$ visits two nodes, then $\overline{\rm{S_1 S_2}}$ and $\overline{\rm{S_3 S_4}}$ visit one node each, so we can impose (from the beginning) that $\overline{\rm{S_1 S_2}}$ visits two nodes of $H(2,k)$.

Such trail, $\mathcal{P}(3)$, can be built starting from the just described trail $\mathcal{P}(2)$, by simply adding a fourth generic Steiner point whose coordinates satisfy
\begin{equation}
    {\rm S}_4 = {\rm S}_3 + \bar{t_3} \cdot \vv{ {\rm S}_3{\rm V}_4},
\end{equation}
so we have
\begin{equation}
    \begin{split}
        x_i({\rm S}_4) = & \bar{t_1} \cdot x_i({\rm V}_2) \cdot (1-\bar{t_2}) + \bar{t_2} \cdot x_i({\rm V}_3) + \bar{t_3} \cdot (x_i({\rm V}_4)-(\bar{t_1} \cdot x_i({\rm V}_2) \cdot (1-\bar{t_2}) + \bar{t_2} \cdot x_i({\rm V}_3))) \\
        =&(\bar{t_1} \cdot x_i({\rm V}_2) \cdot (1-\bar{t_2}) + \bar{t_2} \cdot x_i({\rm V}_3)) \cdot (1-\bar{t_3}) + \bar{t_3} \cdot x_i({\rm V}_4),
    \end{split}
\end{equation}

\vspace{-6mm} \noindent with $\bar{t_3} =1$.
\vspace{2mm}

Before moving on to the segment $\overline{{\rm S}_3{\rm S}_4}$, let us make some considerations for a better understanding of the nature of the next step of the present proof.

We have that $\bar{t_1} \geq 1$ and $\bar{t_2}> 1$, since $\bar{t_2} = 1$ would imply ${\rm S}_3 = {\rm V}_3$. It follows that $\overline{{\rm S}_3{\rm S}_4}$ would visit ${\rm V}_2$ and ${\rm V}_3$, whereas it could not visit other nodes since the set $H(2,k)$ has not more than $2$ collinear nodes.

Under these constraints, the following results are obtained so that we can use them to find the solutions of Equation (\ref{xiV4}).
\begin {enumerate}
    \item $\bar{t_1} \cdot (1-\bar{t_2}) <0$, since $\bar{t_1}> 0$ and $(1-\bar{t_2}) <0$.
    \item $\bar{t_1} \cdot (1-\bar{t_2}) + \bar{t_2} <1$, since $\bar{t_1}=1$ implies $\bar{t_1} \cdot (1-\bar{t_2}) + \bar{t_2} =1$ and\\ $\frac{\partial}{\partial \bar{t_1}}(\bar{t_1} \cdot (1-\bar{t_2}) + \bar{t_2})<0$ $\forall \; \bar{t_1}>1,\bar{t_2}>1$. 
\end {enumerate}

Now, we consider the segment $\overline{{\rm S}_3{\rm S}_4}$. By disregarding the node ${\rm V}_4$, obtained by imposing $t_3 = 1$, we verify that it does not exist any unvisited node ${\rm V}_j$ of $H(2,k)$, belonging to $\overline{{\rm S}_3{\rm S}_4}$.

Thus,
 \begin{equation}\label{xiV4}
 \resizebox{.999\hsize}{!}{$
    \begin{cases}
    \bar{t_1}\cdot x_i({\rm V}_2) \cdot(1-\bar{t_2}) + \bar{t_2}\cdot x_i({\rm V}_3) + t_3\cdot (x_i({\rm V}_4)-(\bar{t_1}\cdot x_i({\rm V}_2) \cdot(1-\bar{t_2}) + \bar{t_2}\cdot x_i({\rm V}_3)))=x_i({\rm V}_j) \\
    \bar{t_1}>1\\
    \bar{t_2}>1\\
    0<t_3<1\\
    \end{cases}$}\hspace{-3.5mm}.
\end{equation}

Hence,
 \begin{align}
    & x_i({\rm V}_2)=x_i({\rm V}_3)=x_i({\rm V}_4)=x_i({\rm V}_j)=0 \Rightarrow t_3\in \mathbb{R}:0<t_3<1, \; & \label{V=V=0}\\
    &x_i({\rm V}_2)=x_i({\rm V}_3)=1,x_i({\rm V}_4)=x_i({\rm V}_j)=0 \wedge \bar{t_1}=\frac{\bar{t_2}}{\bar{t_2}-1} \Rightarrow t_3\in \mathbb{R}:0<t_3<1, \label{V2=V3}\\
    &x_i({\rm V}_2)=x_i({\rm V}_4)=0,x_i({\rm V}_3)=x_i({\rm V}_j)=1  \Rightarrow t_3=\frac{\bar{t_2}-1}{\bar{t_2}}, \label{V2=V4}\\
    &x_i({\rm V}_2)=x_i({\rm V}_3)=x_i({\rm V}_4)=1,x_i({\rm V}_j)=0 \Rightarrow t_3=\frac{\bar{t_1}(1-\bar{t_2})+\bar{t_2}}{\bar{t_1}(1-\bar{t_2})+\bar{t_2}-1}. \label{V2=V3=V4}
\end{align}

There cannot be two indices $i, i'$ such that $(x_i({\rm V}_2)=x_i({\rm V}_3)=1,x_i({\rm V}_4)=0)\wedge (x_i'({\rm V}_2)=x_i'({\rm V}_3)=x_i'({\rm V}_4)=1,x_i'({\rm V}_j)=0)$ (see References (\ref{V2=V3})\&(\ref{V2=V3=V4})), since by imposing $\bar{t_1} = \frac{\bar{t_2}}{\bar{t_2}-1}$ we obtain $t_3 = 0$.

The uniqueness of the indices that simultaneously verify (\ref{V=V=0}),(\ref{V2=V3})\&(\ref{V2=V4}) implies that ${\rm V}_1$ is visited twice, while the uniqueness of the indices that simultaneously verify (\ref{V=V=0}),(\ref{V2=V4})\&(\ref{V2=V3=V4}) implies that ${\rm V}_2$ is visited twice.

Therefore, it is not possible to join more than $4$ nodes of the given set with a polygonal chain consisting of only $3$ links, and this concludes the proof of Lemma \ref{lemma 1}.
\end{proof}

By invoking Lemma \ref{lemma 1}, we can easily prove the following theorem.

\begin{theorem}\label{lower bound}
Let $k\in \mathbb{N}-\{0,1\}$ be given. The link-length of the covering trail $\mathcal{P}(h(2,k))$ for the set $H(2,k)$ satisfies $h(2,k)\geq 3\cdot 2^{k-2}$.
\end{theorem}

\begin{proof}
There is a total of $2^k$ nodes to be visited. By Lemma \ref{lemma 1}, we can join a maximum of $4$ nodes with a polygonal chain of link-length $3$. Since $H(2,k)$ has $\frac{2^k} {4}$ groups of $4$ nodes, the $4$ nodes of each of these groups require a minimum of $3$ segments to be visited. Therefore, $h(2,k) \geq 3 \cdot 2^{k-2}$.
\end{proof}

Now, we need to find the shortest possible covering path, $\mathcal{P} (h(2,k))$. For this purpose, it is possible to prove the following result.

\begin{theorem}\label{upper bound 1}
Given $H(2,k)$ with $k\in \mathbb{N}-\{0,1\}$, it is always possible to construct a covering cycle $\mathcal{P}(h(2,k))$ of link-length $h(2,k) = 3 \cdot 2^{k-2}$.
\end{theorem}

\begin{proof}

It is possible to create an algorithm that generates a covering circuit for $H(2,k)$ whose link-length exactly coincides with the lower bound stated by Theorem \ref{lower bound}. The algorithm is valid for any finite number of dimensions.

\vspace{2mm}

First of all, we notice that it is always possible to join the four nodes of a rectangle with a covering circuit of link-length $3$. In fact, given the set $\{ \{0,1\} \times \{0,b\} \}$, we can have the covering circuit $({\rm S_1})$-$({\rm S_2})$-$({\rm S}_3)$-$({\rm S}_{1})$, with the elements of the set $\{{\rm S_1},{\rm S_2},{\rm S_3} \}$ given by
\begin{align*}
    &{\rm S}_1\equiv \left(\frac{b}{2},\frac{3}{2} \right);\; & & \\
    &{\rm S}_2 \equiv (-b,0)\; & {\rm S}_1+\frac{1}{3}\cdot \vv{{\rm S}_1{\rm S}_2}=(0,1); & \nonumber\\
    &{\rm S}_3\equiv (0,2\cdot b)\; & {\rm S}_2+\frac{1}{3}\cdot\vv{{\rm S}_2{\rm S}_3}=(0,0), & \;\;\; \; {\rm S}_2+\frac{2}{3}\cdot \vv{{\rm S}_2{\rm S}_3}=(b,0); \nonumber \\
    &{\rm S}_4\equiv {\rm S}_1 \; & {\rm S}_3+\frac{2}{3}\cdot\vv{{\rm S}_3{\rm S}_4}=(b,1). & \nonumber
\end{align*}

We consider the sheaf of planes that have in common the line $r:={\rm C}+t\cdot \vv{e_k}$\hspace{1mm}, where \linebreak ${\rm C}\equiv \left(\frac{1}{2}, \frac{1}{2},\dots,\frac{1}{2},0 \right)$ and $\vv{e_k}:= (0,0, \dots,0,1)$ is the second vector of the canonical basis. 
These planes have parametric equation ${\rm C}+t\cdot \vv{e_k} + u\cdot \vv{s}$ with $\vv{e_k}$ and $\vv{s}$ being linearly independent vectors.

Let $\vv{s_l}\in \{\{-\frac{1}{2}\}\times\{-\frac{1}{2},\frac{1}{2}\}\times\{-\frac{1}{2},\frac{1}{2}\}\times \cdots \times\{-\frac{1}{2},\frac{1}{2}\}\times\{0\}\}\subseteq \mathbb{R}^{k-2}$ be the vector such that $l \in \mathbb{N} : l<2^{k-2}$.

If $t=u=1$, then we obtain the point ${\rm V}_{4l+1} \in H(2,k)$ such that\\ ${\rm V}_{4l+1}\equiv (0,x_2({\rm V}_{4l+1}),x_3({\rm V}_{4l+1}),\dots,x_{k-1}({\rm V}_{4l+1}),1)$, where
\begin{equation}
    \begin{cases}
       x_i({\rm V}_{4l+1})= 1 & \text{if} \; x_i(\vv{{ s}_{l}})=+\frac{1}{2}\\
       x_i({\rm V}_{4l+1})= 0 & \text{if} \; x_i(\vv{{ s}_{l}})=-\frac{1}{2}
    \end{cases}.
\end{equation}

If $t=0 \wedge u=1$, then we obtain the point ${\rm V}_{4l+2} \in H(2,k)$ such that\\ ${\rm V}_{4l+2}\equiv (0,x_2({\rm V}_{4l+2}),x_3({\rm V}_{4l+2}),\dots,x_{k-1}({\rm V}_{4l+2}),1)$, where  
\begin{equation}
    \begin{cases}
       x_i({\rm V}_{4l+2})= 1 & \text{if} \; x_i(\vv{{ s}_{l}})=+\frac{1}{2}\\
       x_i({\rm V}_{4l+2})= 0 & \text{if} \; x_i(\vv{{ s}_{l}})=-\frac{1}{2}
    \end{cases}.
\end{equation}

If $t=0 \wedge u=-1$, then we obtain the point ${\rm V}_{4l+3} \in H(2,k)$ such that\\ ${\rm V}_{4l+3}\equiv (0,x_2({\rm V}_{4l+3}),x_3({\rm V}_{4l+3}),\dots,x_{k-1}({\rm V}_{4l+3}),1)$, where  
\begin{equation}
    \begin{cases}
       x_i({\rm V}_{4l+3})= 0 & \text{if} \; x_i(\vv{{ s}_{l}})=+\frac{1}{2}\\
       x_i({\rm V}_{4l+3})= 1 & \text{if} \; x_i(\vv{{ s}_{l}})=-\frac{1}{2}
    \end{cases}.
\end{equation}

If $t=1 \wedge u=-1$, then we obtain the point ${\rm V}_{4l+4} \in H(2,k)$ such that\\ ${\rm V}_{4l+4}\equiv (0,x_2({\rm V}_{4l+4}),x_3({\rm V}_{4l+4}),\dots,x_{k-1}({\rm V}_{4l+4}),1)$, where  
\begin{equation}
    \begin{cases}
       x_i({\rm V}_{4l+4})= 0 & \text{if} \; x_i(\vv{{ s}_{l}})=+\frac{1}{2}\\
       x_i({\rm V}_{4l+4})= 1 & \text{if} \; x_i(\vv{{ s}_{l}})=-\frac{1}{2}
    \end{cases}.
\end{equation}

Being $\{\sigma\}_l:=\{\sigma_l : l=0,1,2,\dots,2^{k-2}-2,2^{k-2}-1\}$ the set of the planes containing $r$ and ${\rm C} + \vv{s_l}$, in total we will have exactly $4$ nodes of $H(2,k)$ lying on each one of the aforementioned planes (i.e., there are $2^{k-2}$ planes $\sigma_l$ such that $0 \leq l < 2^{k-2}$, being $2^{k-2}$ the multisubsets of size $k-2$ from the set $\{\frac{1}{2},-\frac{1}{2}\}$). 

Since each of the $2^{k-2}$ multisubsets is different from all the others, it follows that it does not exist any pair of positive integers $(j,j') : j \neq j'$ such that ${\rm V}_{j} \equiv {\rm V}_{j'}$. Consequently, each plane contains exactly $4$ nodes that do not belong to any other plane $\sigma_l$.

Thus, there are $2^{k-2}$ planes that include $4$ different nodes each, for a total of $2^k$ distinct nodes. Since a $k$-dimensional hypercube has exactly $2^k$ nodes, we conclude that each point lies on one and only one plane $\sigma_l$.

Let $l \in \mathbb{N}_0 : 0 \leq l \leq 2^{k-2}-1$ be given. Then, the four points ${\rm V}_{4l+1}$,${\rm V}_{4l+2}$,${\rm V}_{4l+3}$, and ${\rm V}_{4l+4}$ identify the nodes of a rectangle with base $1$ and height $\sqrt{2 \cdot (k-1)}$. In fact, $\vv{e_2}=(0,1,0,0,\dots,0)$ forms an orthogonal basis with vector $\vv{s}$ that has coordinate $s_2=0$.

We have already proven that the nodes in this orthogonal basis are in position $(0,-1)$; $(0,1)$; $(1,1)$; $(1,-1)$ so that, in this basis, they are vertices of a rectangle of base $1$ and height $2$. Consequently, being $\sqrt{\frac{1}{2} \cdot (k-1)}$ the magnitude of vector $\vv{s}$, we get a height of $2 \cdot \sqrt{\frac{1}{2} \cdot (k-1)}=\sqrt{2 \cdot (k-1)}$ in the canonical basis of the space $E$.

Finally, we have $2^{k-2}$ rectangles, whose vertices can be covered by $2^{k-2}$ covering circuits of link-length $3$ (see Figures 1\&2).

\begin{figure}[H]
\begin{center}
\includegraphics[width=15cm]{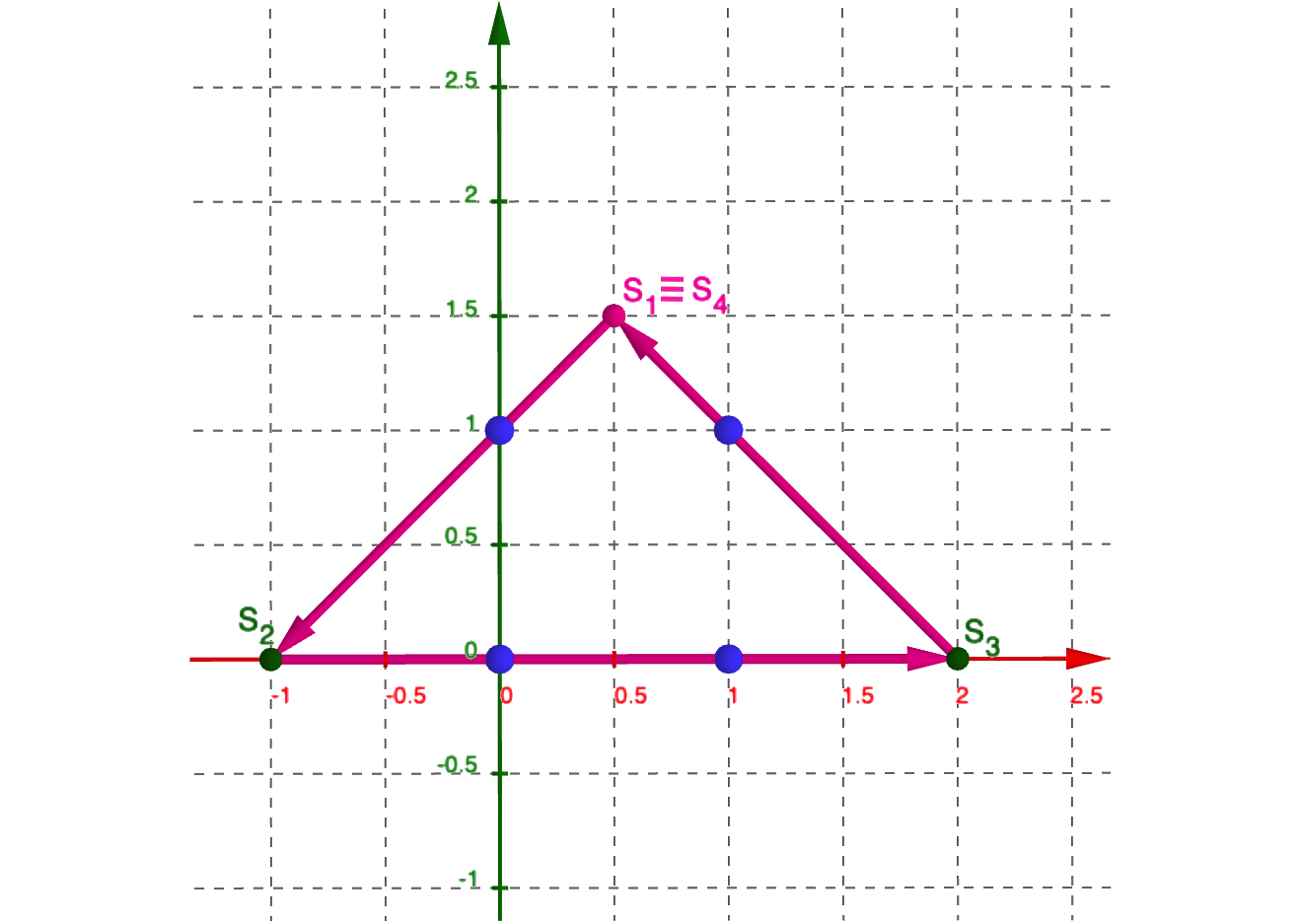}
\end{center}
\caption{The minimum-link perfect covering cycle $\mathcal{\bar{C}}(h(2,2)):=\left(\frac{1}{2},\frac{3}{2}\right)$-$(-1,0)$-$(2,0)$-$\left(\frac{1}{2},\frac{3}{2}\right)$\\ joins all the nodes of $H(2,3)$ (picture realized with GeoGebra \cite{Geogebra:32}).}
\label{fig:Figure_1}
\end{figure}

\begin{figure}[H]
\begin{center}
\includegraphics[width=12cm]{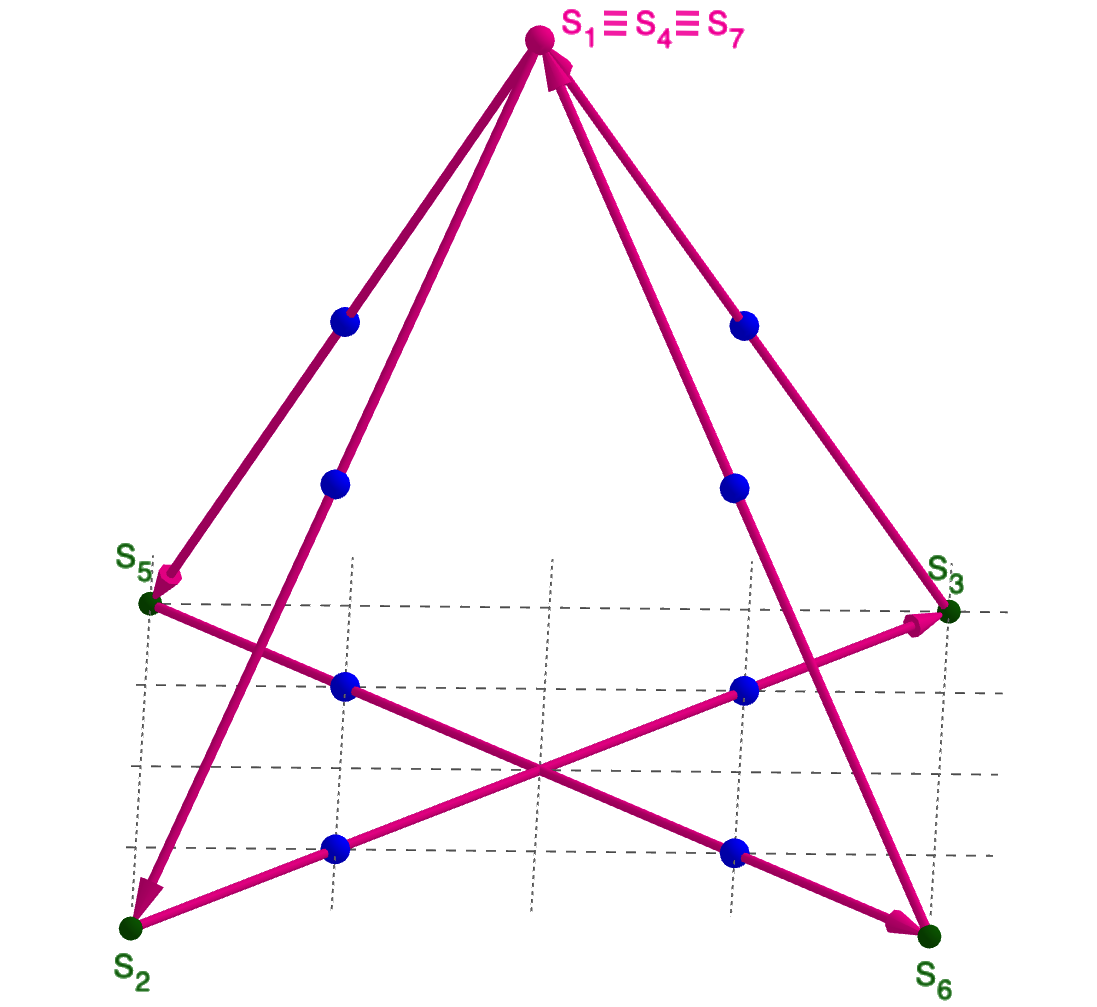}
\end{center}
\caption{The minimum-link closed polygonal chain $\mathcal{P}(6):=\left(\frac{1}{2},\frac{1}{2},2 \right)$-$\left(-\frac{1}{2},-\frac{1}{2},0 \right)$-$(\frac{3}{2},\frac{3}{2},0)$-$\left(\frac{1}{2},\frac{1}{2},2 \right)$-$(-\frac{1}{2},\frac{3}{2},0)$-$(\frac{3}{2},-\frac{1}{2},0)$-$\left(\frac{1}{2},\frac{1}{2},2 \right)$\\ visits all the nodes of $H(2,3)$ once and only once (picture realized with GeoGebra \cite{Geogebra:32}).}
\label{fig:Figure_2}
\end{figure}

Lastly, using the described covering circuits, we get a circuit that starts and ends at a point that lies on the generating line of the sheaf of planes that contains all the planes belonging to $\{\sigma\}_l$.

Thus,
\vspace{-2mm}
\begin{equation}
 \resizebox{.99\hsize}{!}{$
\{{\rm V}_{4l+1}, {\rm V}_{4l+2}, {\rm V}_{4l+3}, {\rm V}_{4l+4}\} \hspace{-0.1mm} \subset \hspace{-0.1mm} P_m(4) \Rightarrow \exists! \hspace{0.8mm} l \in \{0,1,\dots,2^{k-2}-1\} : \left({\rm S}_{3l+1}\right)\hspace{-1mm}\textnormal{-}\hspace{-1mm}\left({\rm S}_{3l+2}\right)\hspace{-1mm}\textnormal{-}\hspace{-1mm}\left({\rm S}_{3l+3}\right)\hspace{-1mm}\textnormal{-}\hspace{-1mm}\left({\rm S}_{3l+4}\right) \hspace{-0.1mm}\subset \hspace{-0.1mm} \{\sigma\}_l $}.
\end{equation}

As shown in Figure 3, we obtain a covering cycle for $H(2,k)$ by the repetition, for every $l \in \{0, 1, 2, \dots, 2^{k-2}-1\}$, of the covering circuit described by
\begin{align}
    &{\rm S}_{3l+1}\equiv {\rm C}+\frac{3}{2}\cdot \vv{e_1};  & \nonumber \\
    &{\rm S}_{3l+2}\equiv {\rm C} + 3\cdot \vv{s_l}&  {\rm S}_{3l+1}+\frac{1}{3}\cdot \vv{{\rm S}_{3l+1}{\rm S}_{3l+2}}\equiv{\rm V}_{4l+1}; \nonumber \\
    &{\rm S}_{3l+3}\equiv {\rm C}- 3 \cdot \vv{s_l} & {\rm S}_{3l+2}+\frac{1}{3}\cdot\vv{{\rm S}_{3l+2}{\rm S}_{3l+3}}\equiv{\rm V}_{4l+2}, \nonumber \\
    & & {\rm S}_{3l+2}+\frac{2}{3}\cdot\vv{{\rm S}_{3l+2}{\rm S}_{3l+3}}\equiv{\rm V}_{4l+3}; \nonumber \\
    &{\rm S}_{3l+4}\equiv {\rm S}_{3l+1}  & {\rm S}_{3l+3}+\frac{2}{3}\cdot\vv{{\rm S}_{3l+3}{\rm S}_{3l+4}}\equiv{\rm V}_{4l+4}. \nonumber
\end{align}
\begin{figure}[H]
\begin{center}
\includegraphics[width=\linewidth]{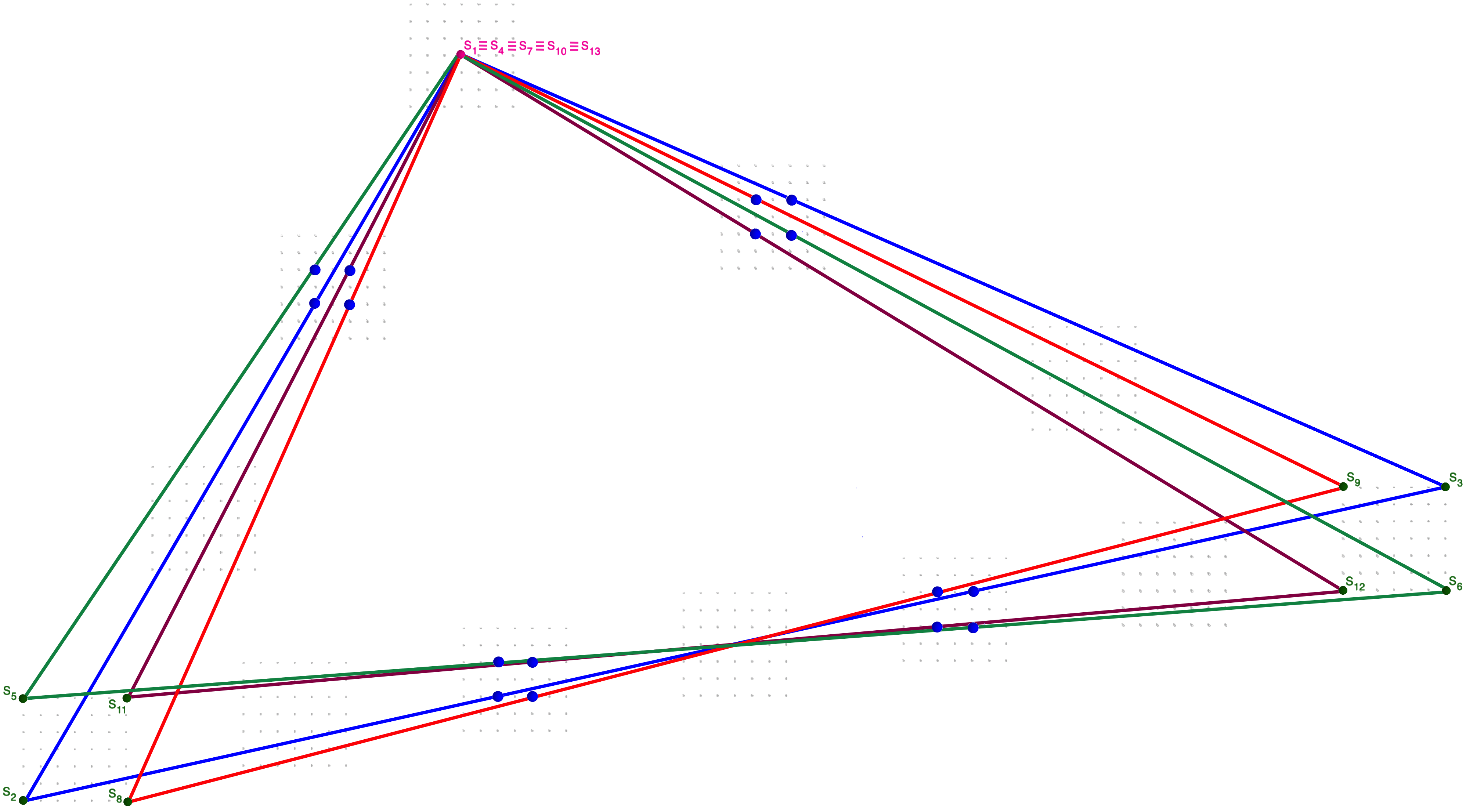}
\end{center}
\caption{The minimum-link closed polygonal chain $\mathcal{P}(12):=\left(\frac{1}{2},\frac{1}{2},\frac{1}{2},\frac{3}{2} \right)$-$\left(-1,-1,-1, 0 \right)$-$(2,2,2,0)$-$\left(\frac{1}{2},\frac{1}{2},\frac{1}{2},\frac{3}{2} \right)$-$(-1,-1,2,0)$-$(2,2,-1,0)$-\\$\left(\frac{1}{2},\frac{1}{2},\frac{1}{2},\frac{3}{2} \right)$-$(-1,2,-1,0)$-$(2,-1,2,0)$-$\left(\frac{1}{2},\frac{1}{2},\frac{1}{2},\frac{3}{2} \right)$-$(-1,2,2,0)$-$(2,-1,-1,0)$-$\left(\frac{1}{2},\frac{1}{2},\frac{1}{2},\frac{3}{2} \right)$\\ joins all the nodes of $H(2,4)$ (picture realized with GeoGebra \cite{Geogebra:32}).}
\label{fig:Figure_4}
\end{figure}

Therefore, we have constructively proven that $h(2,k)\leq 3\cdot 2^{k-2}$, for any $k \in \mathbb{N}-\{0,1\}$.
\end{proof}

\vspace{2mm}
Lastly, we note that it is also possible to generate a covering cycle that does not have coincident Steiner points, except for the first and the last one, following a variation of the previous algorithm, as shown by Corollary \ref{upper bound 2} (see also Figures 4\&5).

\begin{figure}[H]
\begin{center}
\includegraphics[width=\linewidth]{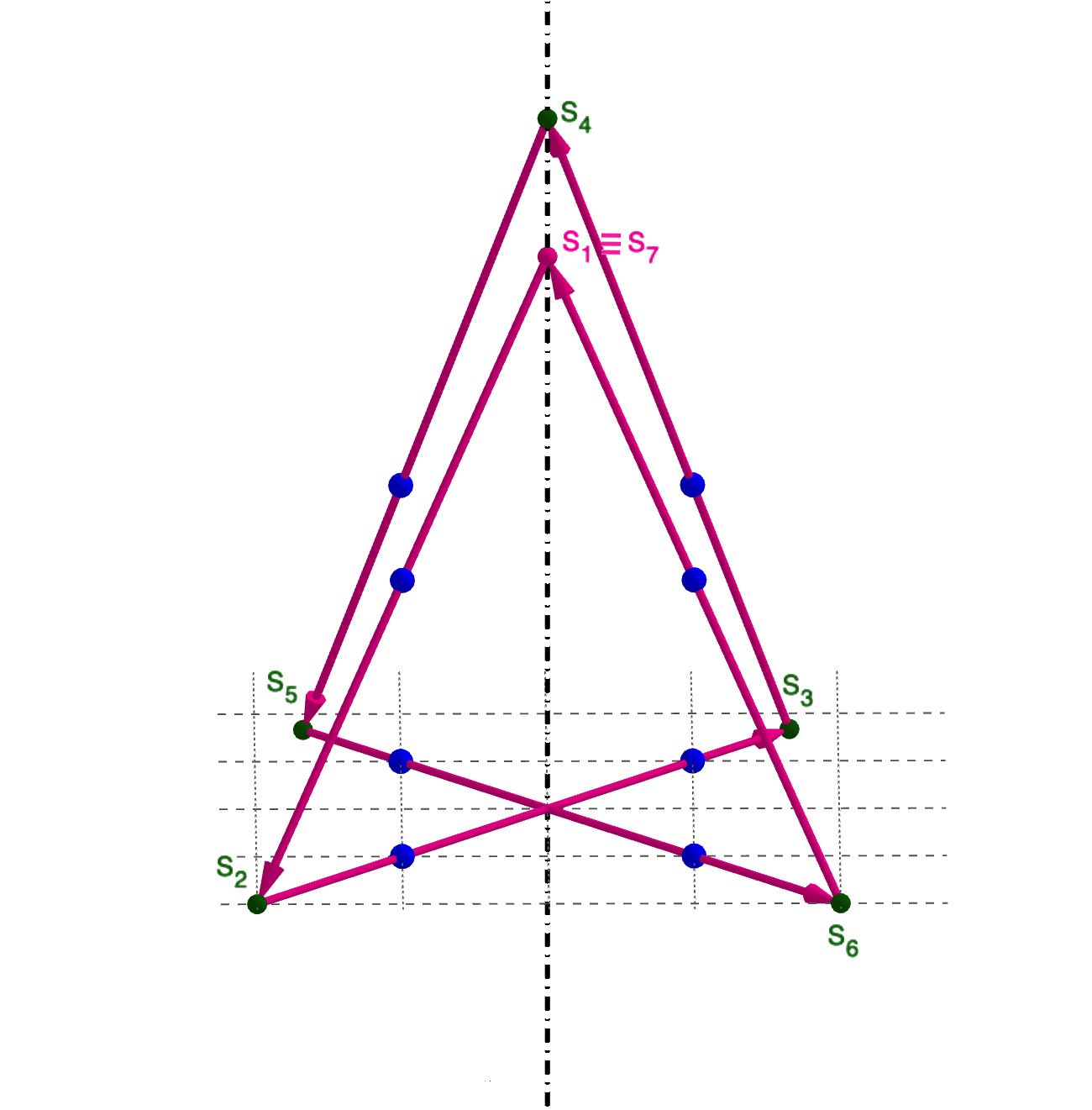}
\end{center}
\caption{The minimum-link perfect covering cycle $\mathcal{\bar{C}}(h(2,3)):=\left(\frac{1}{2},\frac{1}{2},2 \right)$-$\left(-\frac{1}{2},-\frac{1}{2},0 \right)$-$\left(\frac{4}{3},\frac{4}{3},0 \right)$-$\left(\frac{1}{2},\frac{1}{2},\frac{5}{2} \right)$-$\left(-\frac{1}{3},\frac{4}{3},0 \right)$-$\left(\frac{3}{2},-\frac{1}{2},0 \right)$-$\left(\frac{1}{2},\frac{1}{2},2 \right)$\\ joins all the nodes of $H(2,3)$ (picture realized with GeoGebra \cite{Geogebra:32}).}
\label{fig:Figure_3}
\end{figure}

\begin{figure}[H]
\begin{center}
\includegraphics[width=\linewidth]{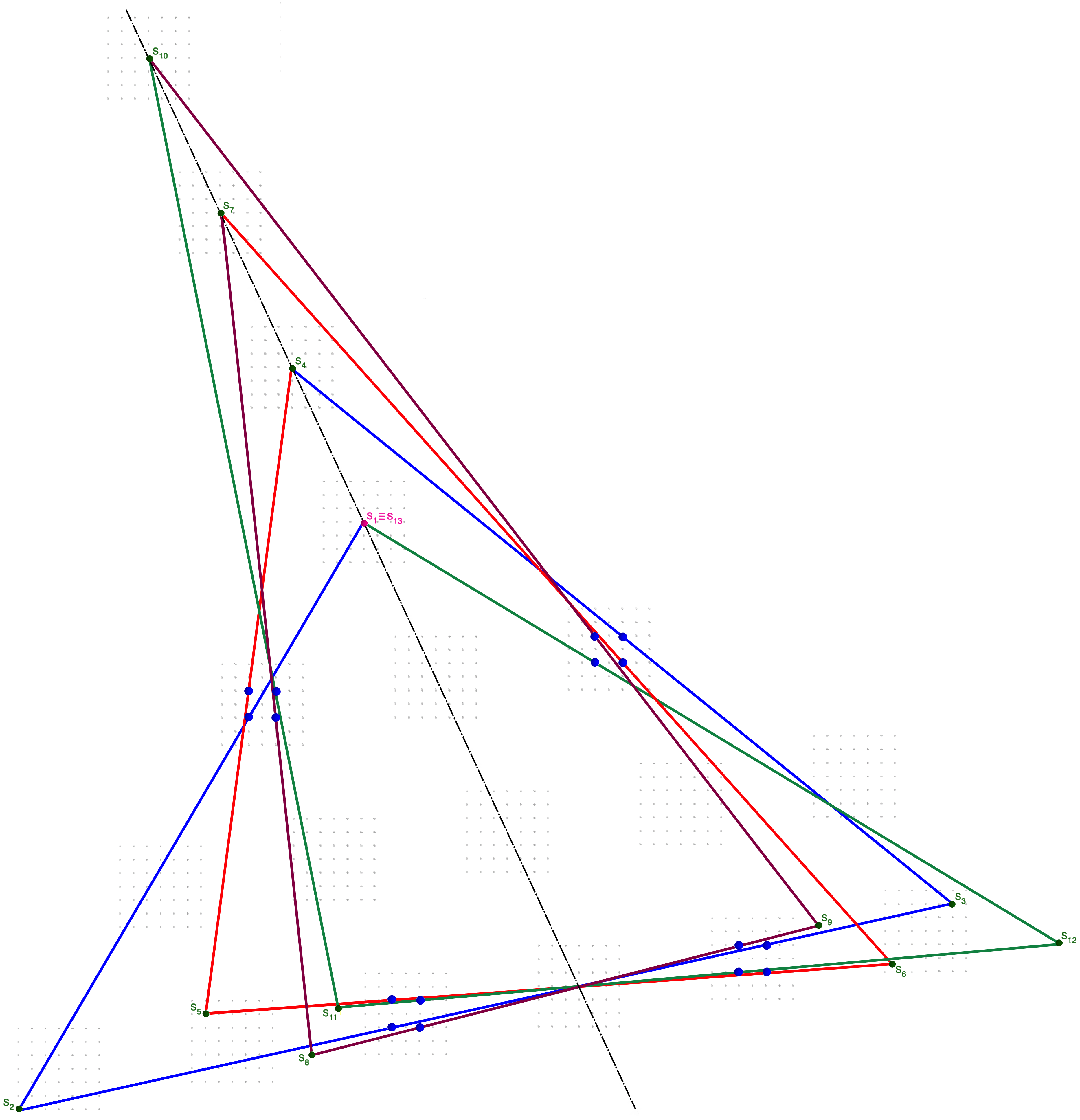}
\end{center}
\caption{The minimum-link perfect covering cycle $\mathcal{\bar{C}}(h(2,4))\hspace{-1mm}:=\hspace{-1mm}\left(\frac{1}{2},\frac{1}{2},\frac{1}{2},\frac{3}{2} \right)$-$\left(-1,-1,-1,0 \right)$-$\left(\frac{3}{2},\frac{3}{2},\frac{3}{2},0 \right)$-$\left(\frac{1}{2},\frac{1}{2},\frac{1}{2},2 \right)$-$\left(-\frac{1}{2},-\frac{1}{2},\frac{3}{2},0 \right)$-$\left(\frac{4}{3},-\frac{1}{3},-\frac{1}{3},0 \right)$-\\$\left(\frac{1}{2},\frac{1}{2},\frac{1}{2},\frac{5}{2} \right)$-$\left(-\frac{1}{3},\frac{4}{3},-\frac{1}{3},0 \right)$-$\left(\frac{5}{4},-\frac{1}{4},\frac{5}{4},0 \right)$-$\left(\frac{1}{2},\frac{1}{2},\frac{1}{2},3 \right)$-$\left(-\frac{1}{4},\frac{5}{4},\frac{5}{4},0 \right)$-$\left(2,-1,-1,0 \right)$-$\left(\frac{1}{2},\frac{1}{2},\frac{1}{2},\frac{3}{2} \right)$\\ visits all the nodes of $H(2,4)$ once and only once (picture realized with GeoGebra \cite{Geogebra:32}).}
\label{fig:Figure_5}
\end{figure}

\begin{corollary} \label{upper bound 2}
Given $H(2,k)$ with $k\in \mathbb{N}-\{0,1\}$, it is always possible to construct a perfect covering cycle $\bar{\mathcal{C}}(h(2,k)) := \{\overline{{\rm{S_1 S_2}}} \cup \overline{{\rm{{S_2 S_3}}}} \cup \dots \cup \overline{{\rm{S}}_{3 \cdot 2^{k-2}
}{\rm{S}}_{1}}\}$ having exactly $3 \cdot 2^{k-2}$ distinct Steiner points and such that ${\rm{S}}_{3 \cdot 2^{k-2} + 1} \equiv {\rm{S}}_{1}$.
\end{corollary}
\begin{proof}
We constructively prove the corollary by following the same approach that has been introduced in the proof of Theorem \ref{upper bound 1}, taking also into account that the Steiner points of the type ${\rm S}_ {3l + 1}$, that lie on the straight line $r$, must have the coordinates $x_{k}({\rm S}_{3l+1})$ different from each other.

We can choose $ x_{k}({\rm S}_{3l+1}):= \frac {l+2} {2}$ to obtain
\begin{align}
    &{\rm S}_{3l+1}\equiv {\rm C}+\frac {l+2} {2}\cdot \vv{e_2}; & \\
    &{\rm S}_{3l+2}\equiv {\rm C}+\left(1+\frac{1}{l}\right)\cdot \vv{s_l} & {\rm S}_{3l+1}+\frac{l+1}{2l+3}\cdot\vv{{\rm S}_{3l+1}{\rm S}_{3l+2}}\equiv{\rm V}_{4l+1}; \nonumber \\
    &{\rm S}_{3l+3}\equiv{\rm C}-\left(1+\frac{1}{l}\right)\cdot \vv{s_l} & {\rm S}_{3l+2}+\frac{l+1}{2l+3}\cdot\vv{{\rm S}_{3l+2}{\rm S}_{3l+3}}\equiv{\rm V}_{4l+2}; \nonumber \\ 
    & &{\rm S}_{3l+2}+\frac{l+2}{2l+3}\cdot\vv{{\rm S}_{3l+2}{\rm S}_{3l+3}}\equiv{\rm V}_{4l+3}; \nonumber \\
    &{\rm S}_{3l+4}\equiv {\rm C}+\frac {l+3} {2}\cdot \vv{e_2} & {\rm S}_{3l+3}+\frac{l+2}{2l+3}\cdot\vv{{\rm S}_{3l+3}{\rm S}_{3l+4}}\equiv{\rm V}_{4l+4}. \nonumber
\end{align}

Therefore, for any given $k \in \mathbb{N}-\{0,1\}$ we have provided a perfect covering cycle, for $H(2,k)$, which is characterized by a link-length of $3 \cdot 2^{k-2}$ and such that no Steiner point is visited more than once, with the only exception of the starting/ending point, ${\rm S}_1 \equiv {\rm S}_{1+3 \cdot 2^{k-2}}$.

This concludes the proof of Corollary \ref{upper bound 2}.
\end{proof}


\section{Conclusion}
Although for any $H(2,k)$ we have constructively shown the existence of perfect covering cycles whose link-length is equal to $3 \cdot 2^{k-2}$, the problem of finding an analogous formula concerning optimal covering paths for any given set $H(n,k):=\{\{0, 1, \ldots, n-1\} \times \{0, 1, \ldots, n-1\} \times \cdots \times \{0, 1, \ldots, n-1\}\} \subset \mathbb{R}^k$, such that $n \geq 4 \wedge k \geq 3$, remains completely open \cite{Ripa:22} (e.g., we can only say that $h(4,3) \in \mathbb\{21, 22, 23\}$ \cite{OEIS:31, Ripa:23}).


\section*{Acknowledgments}
We sincerely thank Luca Onnis for his kind assistance on the initial phase of the present preprint.

\bibliographystyle{plain}
\bibliography{Optimal_cycles_enclosing_all_the_nodes_of_a_k-dimensional_hypercube}

\end{document}